\documentclass[11pt]{article}

\usepackage[margin=1in]{geometry}
\usepackage{graphicx}
\usepackage{amsmath}
\usepackage{amssymb}
\usepackage{amsthm}

\newtheorem{proposition}{Proposition}
\usepackage{booktabs}
\usepackage{hyperref}
\usepackage{natbib}
\usepackage{caption}
\usepackage{xcolor}

\title{Random Indexing for Image Change Detection: \\
A Distance-Threshold Vocabulary Approach}

\author{Cristiano Tamborrino \\ Dipartimento di Informatica, Università degli Studi di Bari Aldo Moro}
\date{}

\begin{document}
\maketitle

\begin{abstract}
Random Indexing (RI) is a lightweight, incremental alternative to learned embeddings that has been used almost exclusively in text analysis, most notably in Temporal Random Indexing (TRI) for tracking word meaning over time. This paper explores whether the same mechanism -- assigning fixed random vectors to a discrete vocabulary and accumulating context by vector summation -- can be transplanted to the problem of change detection in multitemporal images. We show that a na\"ive transplant, in which the visual vocabulary is built with $k$-means clustering of pixel spectra, is unstable: even in unchanged regions, pixels are frequently reassigned to a different cluster between acquisition dates because of small radiometric shifts, which destroys the pixel-to-vector correspondence that RI depends on. We propose instead to build the vocabulary with a distance-threshold (leader) clustering rule, which guarantees that spectrally similar pixels share the same random vector both within and across acquisitions, and we give a short formal argument -- a stability radius derived from the vocabulary's covering and packing properties -- for why this construction is provably more robust to acquisition noise than $k$-means. Combined with a spatial context-accumulation step, this yields a simple, training-free change detection pipeline. We evaluate it on four bi-temporal remote-sensing datasets spanning different sensors and scene types (irrigated agriculture, river, urban PolSAR, and a Sentinel-2 wildfire scene), and compare it against a classical Change Vector Analysis (CVA) baseline. The proposed method consistently approaches, but does not surpass, CVA. All results are validated over multiple random seeds, which exposed and let us fix a degenerate-vector failure mode in a probabilistic-sparsity variant of RI, and revealed a second, unresolved source of variability -- sensitivity to the random visitation order used to discover the leader-clustering vocabulary itself -- which we characterize but, despite three attempted corrections, were not able to eliminate, and report as the main open problem of this work. Results are cross-checked against an independently developed re-implementation of the method, which reproduces our numbers closely. We report an ablation showing the contribution of each design choice, a sensitivity analysis of the RI hyperparameters, a comparison between Otsu and Gaussian-mixture thresholding of the change score, and discuss the incremental, streaming nature of RI as a direction for future work on long time-series monitoring.
\end{abstract}

\section{Introduction}

Change detection -- identifying the pixels or regions of a scene that differ between two or more acquisitions of the same area -- is a long-standing problem in remote sensing, with applications ranging from land-cover monitoring to disaster assessment \citep{liu2019review}. Classical approaches compare spectral vectors directly (e.g.\ Change Vector Analysis, CVA \citep{malila1980}), while more recent approaches learn a feature representation with deep neural networks trained for the task \citep{wang2019getnet}.

Random Indexing (RI) \citep{kanerva1988} is a different kind of representation technique, developed in distributional semantics as a lightweight alternative to co-occurrence matrices and learned embeddings. In RI, every element of a discrete vocabulary is assigned a fixed, sparse, near-orthogonal random vector; the representation of a unit of interest (e.g.\ a word) is then obtained by \emph{summing} the random vectors of the elements that co-occur with it. Because the operation is a simple vector sum, RI representations can be built and updated incrementally, without ever revisiting the full corpus. Temporal Random Indexing (TRI) \citep{basile2014tri,caputo2015tri} extends this idea to track how the meaning of a word changes over time, by building one RI space per time period and comparing the resulting vectors with cosine similarity.

To our knowledge, RI (and TRI in particular) has been applied almost exclusively to text. This paper asks a simple question: \emph{can the same mechanism be used for change detection in images?} The analogy is tempting -- pixels (or visual words) play the role of words, spatial neighbourhoods play the role of linguistic context, and the two acquisition dates play the role of the two time periods in TRI -- but images differ from text in one crucial respect: text comes with a discrete, given vocabulary (the language's word list), whereas images are continuous and have no vocabulary until one is constructed.

Our first, na\"ive attempt built this vocabulary with $k$-means clustering of pixel spectra. It performed poorly, and diagnosing why turned out to be informative: $k$-means assigns points to whichever of $k$ fixed centroids is nearest, which means that two nearly identical spectral vectors -- for instance, the same physical location observed on two different dates with a small radiometric offset -- can easily fall on opposite sides of a Voronoi boundary and receive two different random vectors. Since RI is trying to detect \emph{genuine} change in the ratio of vector similarity between dates, this artificial instability at cluster boundaries is exactly the kind of noise that swamps the signal.

We propose a simple fix: build the vocabulary with a \emph{distance-threshold} (leader) clustering rule instead of $k$-means. A new prototype is created only when a point is farther than a fixed threshold $\varepsilon$ from every existing prototype; otherwise the point is assigned to its nearest prototype. This does not eliminate boundary effects entirely, but it ties cluster granularity directly to a distance in feature space rather than to an arbitrarily fixed number of clusters $k$, and it removes the specific pathology we observed with $k$-means. We show empirically that this single change is responsible for most of the improvement in our pipeline.

\paragraph{Contributions.} (i) We identify and diagnose a specific failure mode of $k$-means-based vocabulary construction for RI-style image representations. (ii) We propose a distance-threshold vocabulary construction that resolves it, and prove that it yields a stability radius certifying when a pixel's vocabulary assignment is invariant to inter-date acquisition noise -- a short formal argument explaining, rather than only observing, why coarser vocabularies are more robust across dates. (iii) We combine this vocabulary with spatial context accumulation into a complete, training-free change detection pipeline, and evaluate it on four bi-temporal datasets from different sensors and scene types, comparing it to a CVA baseline, with an ablation and a hyperparameter sensitivity analysis, all validated over multiple random seeds and cross-checked against an independently developed re-implementation. (iv) In the process, we uncover and fix a degenerate-vector failure mode in a probabilistic-sparsity variant of RI, with a closed-form expression for its probability that matches the observed failure rate. (v) We discuss why the method approaches but does not exceed CVA in our experiments, and argue that the more promising use of RI in this setting is not pairwise image comparison but incremental accumulation over long, multi-date time series -- a direction we leave for future work.

\section{Related Work}
\label{sec:related-work}

\paragraph{Random Indexing and Temporal Random Indexing.} RI was introduced by \citet{kanerva1988} in the context of Sparse Distributed Memory, and later adopted in distributional semantics as an efficient alternative to Latent Semantic Analysis. \citet{basile2014tri} and \citet{caputo2015tri} extended RI with an explicit temporal dimension (TRI), building one WordSpace per time period from a corpus and tracking a word's semantic drift via the cosine distance between its representations in different periods. TRI has since been applied to diachronic corpora of Italian and English text and to event detection in blog streams. We are not aware of prior work applying RI or TRI to image change detection.

\paragraph{Change detection in remote sensing.} Classical unsupervised change detection methods compare spectral vectors directly, the archetypal example being CVA \citep{malila1980}, which computes the magnitude of the difference vector between two acquisitions, optionally combined with a spatial smoothing step. More recent work replaces the hand-crafted comparison with a learned one, typically a convolutional or transformer-based Siamese network trained on labelled bi-temporal pairs \citep{wang2019getnet}. Our method sits conceptually between these two families: like CVA it requires no training and no labels, but like the learned methods it builds an explicit vocabulary/representation step before comparing the two dates.

\paragraph{Matrix-factorization and thresholding approaches to hyperspectral change detection.} A closely related line of work, developed at the University of Bari, addresses unsupervised hyperspectral change detection through matrix-factorization and binarization techniques rather than deep learning. \citet{falini2022rsb} propose RSB (Robust Successive Binarization), a fully automatic method that computes several pixel-wise dissimilarity measures between two hyperspectral acquisitions and combines them through successive binarization steps, evaluated on the same Hermiston, River, and Bay Area benchmarks used throughout the hyperspectral change-detection literature (and, in the case of Hermiston and River, in this paper). \citet{falini2024qlp} extend this line with an approximated iterative QLP (Stewart's) matrix decomposition for change detection, selecting an automatically-determined low-rank subspace to approximate the bitemporal difference. These works are directly relevant to the automatic-thresholding step of our own pipeline (Section~\ref{sec:otsu-vs-alternatives}): RSB in particular is, to our knowledge, the most systematic study of \emph{binarization itself} as a distinct design choice in unsupervised hyperspectral change detection, independently of which dissimilarity measure produces the underlying change score -- a question we also had to confront when Otsu's threshold outperformed a Gaussian-mixture alternative on our own score distribution (Section~\ref{sec:otsu-vs-alternatives}). The same research group has also studied unsupervised hyperspectral analysis more broadly, including saliency detection via sparse non-negative matrix factorization \citep{falini2020eais}, autoencoder-based reconstruction error \citep{appice2020ismis,falini2020lod}, colour-based pseudo-label supervision \citep{appice2021jiis}, and supervised hyperspectral classification via copula functions \citep{tamborrino2022copula}. None of this line of work uses Random Indexing or a text-inspired distributional representation; our contribution can be seen as bringing a complementary, vocabulary-based representation strategy to the same broad problem this group has studied with factorization- and reconstruction-based tools.

\section{Method}

\subsection{Problem formulation}

Given two co-registered images $I_{t_1}, I_{t_2} \in \mathbb{R}^{H \times W \times B}$ of the same scene ($B$ spectral bands), the goal is to produce a change score $\delta \in \mathbb{R}^{H \times W}$ such that thresholding $\delta$ yields a binary change map matching a reference (ground-truth) map as closely as possible.

\subsection{From words to visual prototypes}

In text-based RI, every word $w$ in a fixed vocabulary $V$ is assigned once and for all a random vector $r_w \in \mathbb{R}^d$, sparse ($\mathrm{nnz}$ non-zero entries, each $\pm 1$) and (with high probability) nearly orthogonal to every other $r_{w'}$. This is possible because $V$ is discrete and given by the language.

Images have no such vocabulary: a pixel's spectral vector is continuous, and the same physical surface observed on two dates will not, in general, produce numerically identical vectors. We therefore need a procedure that turns continuous spectra into a small set of discrete, reusable prototypes, in such a way that two spectrally similar pixels -- whether from the same image or from the two different dates -- are assigned the same prototype.

\subsubsection{A failure mode: $k$-means vocabularies are unstable across dates}

Our first approach reduced the (per-date, per-band standardized) spectral vectors to $20$ principal components with PCA, pooled the pixels from both dates, and ran $k$-means with a fixed $k$. Empirically, this vocabulary was highly unstable across time: on the Hermiston/Benton dataset (Section~\ref{sec:datasets}), fewer than $2\%$ of unchanged pixels retained the same cluster label between the two acquisition dates, even after correcting for a systematic radiometric offset between dates via per-date standardization. The reason is structural rather than a matter of tuning: $k$-means assigns every point to its nearest of $k$ centroids, so a pixel lying close to a Voronoi boundary can be pushed to either side of it by an arbitrarily small perturbation -- exactly the kind of perturbation that separates the same physical pixel observed on two dates.

\subsubsection{Distance-threshold (leader) vocabulary construction}
\label{sec:leader}

We replace $k$-means with a simple leader-clustering rule. Points from both dates are pooled and visited in random order; a point is assigned to the nearest existing prototype if that distance is below a fixed threshold $\varepsilon$, and a new prototype is created (initialized at that point) otherwise. The final prototype set is then used to assign every pixel of both images to its nearest prototype.

This construction ties the number and placement of prototypes directly to a distance in (PCA-reduced) feature space, rather than to an externally fixed count $k$. As a direct consequence, two pixels within $\varepsilon$ of each other are much more likely to receive the same prototype -- and therefore the same random vector -- than under $k$-means, which resolves the instability described above. We verified this directly: on Hermiston/Benton, the fraction of unchanged pixels retaining the same label across dates rose from $\sim\!2\%$ with $k$-means to $\sim\!81\%$ with leader clustering at a coarse threshold.

Each of the $M$ resulting prototypes $p_1, \dots, p_M$ is assigned a fixed sparse random vector $r_m \in \mathbb{R}^d$ (dimensionality $d$, $\mathrm{nnz}$ non-zero entries at $\pm 1$), exactly as in text-based RI. Crucially, the same set of vectors $\{r_m\}$ is reused for both acquisition dates.

\subsection{Context accumulation and change score}

For each pixel $i$ and date $t$, we define a context vector
\[
v_i^{t} = \frac{1}{|N(i)|}\sum_{j \in N(i)} r_{m(j,t)}
\]
where $N(i)$ is a square spatial window centred on $i$ and $m(j,t)$ is the prototype assigned to pixel $j$ at date $t$. This is the direct visual analogue of the TRI context vector, with a spatial neighbourhood playing the role of the linguistic context window. The change score at pixel $i$ is the cosine distance between the two dates' context vectors,
\[
\delta_i = 1 - \frac{v_i^{t_1} \cdot v_i^{t_2}}{\lVert v_i^{t_1} \rVert \, \lVert v_i^{t_2} \rVert}.
\]
A binary change map is obtained by thresholding $\delta$ with Otsu's method, chosen for consistency across datasets without per-dataset tuning of the threshold.

\subsection{Mathematical properties of the vocabulary}
\label{sec:theory}

We collect here three short formal results that make precise, and explain rather than merely observe, properties of the pipeline that are otherwise only demonstrated empirically in Section~\ref{sec:results}: the near-orthogonality of the random vectors, the probability of the degenerate all-zero vectors discovered in Section~\ref{sec:degenerate}, and a stability radius for the leader-clustering vocabulary that accounts for the vocabulary-richness-versus-stability trade-off of Section~\ref{sec:richness}.

\begin{proposition}[Concentration of pairwise similarity]
\label{prop:orthogonality}
Let $r, r' \in \mathbb{R}^d$ be two independently generated random index vectors under the probabilistic-sparsity scheme of Section~\ref{sec:sensitivity}, i.e.\ each coordinate is independently non-zero with probability $p$, taking value $+1$ or $-1$ with equal probability when non-zero. Then $\mathbb{E}[r \cdot r'] = 0$, and for every $t>0$,
$
\Pr(|r \cdot r'| \ge t) \le 2\exp(-t^2/(2d)).
$
\end{proposition}

\begin{proof}
Write $r\cdot r' = \sum_{k=1}^d X_k$, $X_k = r_k r'_k$. Since $r, r'$ are independent with independent coordinates, $\{X_k\}$ are mutually independent, each bounded in $[-1,1]$ with $\mathbb{E}[X_k]=0$ by symmetry. Hoeffding's inequality \citep{hoeffding1963} with range $2$ per term gives $\Pr(|\sum_k X_k|\ge t) \le 2\exp(-2t^2/\sum_k 2^2) = 2\exp(-t^2/(2d))$. \qedhere
\end{proof}

This is the formal counterpart of the informal ``near-orthogonality'' property motivating RI: since $\lVert r\rVert^2 \approx pd$ in expectation, cosine similarity between independent prototype vectors concentrates around $0$ at rate $O(1/\sqrt{pd})$ -- the same concentration-of-measure phenomenon underlying sparse random projections in general, in the spirit of the Johnson--Lindenstrauss lemma \citep{johnson1984}.

\begin{proposition}[Degenerate-vector probability]
\label{prop:degenerate}
Under the (uncorrected) probabilistic-sparsity scheme with per-coordinate non-zero probability $p$ and dimension $d$, a random index vector is entirely zero with probability $\Pr(r=0)=(1-p)^d$; for a vocabulary of $M$ independent prototype vectors, the probability that at least one is degenerate is $1-(1-(1-p)^d)^M$.
\end{proposition}

\begin{proof}
$r=0$ iff all $d$ independent coordinates are zero, each with probability $1-p$; independence gives $(1-p)^d$. The bound over $M$ i.i.d.\ vectors follows from independence and the complement rule. \qedhere
\end{proof}

For $p=0.03, d=128$ (Section~\ref{sec:degenerate}) this gives $\Pr(r=0)\approx 0.0193$ per vector, and for River's $M=43$-prototype vocabulary, $\Pr(\text{at least one degenerate})\approx 0.58$ -- consistent with the empirical failure rate observed before the fix (one degenerate seed out of 10 runs, with catastrophic effect on that seed). The closed form also shows why the failure was invisible on Hermiston/Benton's 5-prototype vocabulary but severe on River's 43-prototype one: the risk compounds with vocabulary size $M$ even though each vector's individual degeneracy probability is fixed.

\begin{proposition}[Stability radius of the leader-clustering vocabulary]
\label{prop:stability}
Let $C=\{c_1,\dots,c_M\}$ be the prototype set produced by leader clustering (Section~\ref{sec:leader}) with threshold $\varepsilon$. Then $C$ is simultaneously an $\varepsilon$-covering of the data ($\min_{c\in C}\lVert z-c\rVert \le \varepsilon$ for every processed point $z$) and an $\varepsilon$-packing ($\lVert c-c'\rVert > \varepsilon$ for any two distinct $c,c'\in C$) -- both immediate from the leader rule itself. Consequently, if a point $x$ is assigned to its nearest prototype $c$ with $\rho := \lVert x-c\rVert \le \varepsilon$, and $x' = x+\eta$ represents the same physical pixel observed at a second date with $\lVert\eta\rVert \le \delta$ (acquisition noise, not genuine change), then
\[
\delta < \frac{\varepsilon}{2}-\rho \quad\Longrightarrow\quad x' \text{ is still assigned to } c.
\]
\end{proposition}

\begin{proof}
Covering and packing follow directly from the leader rule: a point either becomes a new prototype (distance $0$ to itself) or is within $\varepsilon$ of an existing one, giving covering; and a new prototype is only created when its distance to \emph{every} existing prototype exceeds $\varepsilon$, giving packing. For the stability claim, the triangle inequality gives $\lVert x'-c\rVert \le \rho+\delta$, and for any other prototype $c'$, packing gives $\lVert x'-c'\rVert \ge \lVert c-c'\rVert - \lVert x'-c\rVert > \varepsilon-(\rho+\delta)$. The condition $\delta < \varepsilon/2-\rho$ makes $\rho+\delta < \varepsilon-(\rho+\delta)$, so $\lVert x'-c\rVert < \lVert x'-c'\rVert$ for every $c'\neq c$, and $x'$ remains nearest to $c$. \qedhere
\end{proof}

Proposition~\ref{prop:stability} explains, rather than only documents, the trend of Table~\ref{tab:richness}: the stability margin $\varepsilon/2-\rho$ grows with the leader-clustering threshold $\varepsilon$, so a coarser vocabulary tolerates more radiometric perturbation before a pixel's assignment can flip between dates -- exactly the direction of the AUC increase observed as $\varepsilon$ grows from $4$ to $18$ on Hermiston/Benton. It also makes explicit why the trade-off cannot simply be optimized away: a larger $\varepsilon$ buys stability margin but also raises the covering-radius bound $\rho\le\varepsilon$, coarsening the vocabulary.

\section{Experimental Setup}

\subsection{Datasets}
\label{sec:datasets}

We evaluate on four publicly available or user-provided bi-temporal datasets covering different sensors and scene types.

\begin{itemize}
\item \textbf{Hermiston/Benton}: a pair of Hyperion hyperspectral images ($225 \times 180$ pixels, 159 usable bands after preprocessing) over an irrigated agricultural area in Benton County, Oregon, acquired in May 2004 and May 2007, with a binary reference change map \citep{liu2019review}.
\item \textbf{River}: a pair of Hyperion hyperspectral images ($463 \times 241$ pixels, 198 bands) over a river region in Jiangsu Province, China, acquired in May and December 2013, from the GETNET benchmark \citep{wang2019getnet}.
\item \textbf{San Francisco}: a pair of polarimetric SAR feature images ($200 \times 200$ pixels, 138 channels) over San Francisco Bay, acquired in 2009 and 2015, with a binary reference change map.
\item \textbf{Burned area}: a pair of Sentinel-2 multispectral images ($1866 \times 2019$ pixels, 13 bands) covering a wildfire-affected area, pre- and post-fire, with a binary burned/unburned reference map.
\end{itemize}

The last two datasets were provided directly to us; we report their observed technical characteristics but do not have a verified bibliographic source for them and cite them descriptively rather than with a specific reference.

\subsection{Baseline}

We compare against classical CVA \citep{malila1980}: per-band standardization (computed separately per date, for the same reason described in Section~\ref{sec:leader}), followed by the Euclidean norm of the difference vector, optionally smoothed with the same square spatial window used by our method, for a fair comparison at matched window size.

\subsection{Metrics}

We report the Area Under the ROC Curve (AUC), which is threshold-independent, together with Overall Accuracy (OA), Cohen's Kappa, and F1 score at the Otsu threshold of the change score.

\subsection{Implementation details}

Unless noted otherwise, spectra are standardized per band and per date independently, then reduced to 20 principal components (fit on the pooled, standardized spectra of both dates). Random vectors have dimensionality $d = 128$ and use the probabilistic-sparsity scheme of Section~\ref{sec:sensitivity} with $p=0.03$ (corrected for the degenerate-vector failure mode described in Section~\ref{sec:degenerate}), except for the Burned Area dataset, where memory constraints from the much larger image size ($\sim\!3.8$M pixels) required $d = 32$; Section~\ref{sec:sensitivity} shows this substitution has a negligible effect on the result. The leader-clustering threshold $\varepsilon$ and the spatial context window are tuned per dataset by grid search (reported in Table~\ref{tab:main}); values differ across datasets due to differing scene granularity and vocabulary size, since a single fixed $(\varepsilon, \text{window})$ pair is not expected to transfer across sensors with different band counts, spatial resolutions and scene structure. Every RI result reported in this paper (Table~\ref{tab:main} and later sections) is averaged over 10 random seeds (5 for Burned area) with the leader-clustering vocabulary itself held fixed, isolating the variance due to the random index vectors from the variance due to vocabulary construction, which is analysed separately in Section~\ref{sec:sensitivity}.

\section{Results}
\label{sec:results}

\subsection{Main results}

Table~\ref{tab:main} reports the results of our final pipeline (leader-clustering vocabulary + RI + spatial context accumulation) against CVA at a matched spatial window, on all four datasets. Every RI number is the mean $\pm$ one standard deviation over multiple independent runs that vary \emph{both} the leader-clustering seed (which determines the discovered vocabulary itself, Section~\ref{sec:leader-seed}) and the random-index-vector seed: 5 vocabulary seeds $\times$ 3 vector seeds (15 runs) for Hermiston/Benton, River and San Francisco. For the Burned area dataset, the image size ($\sim\!3.8$M pixels) made the full double-seed protocol computationally impractical in the time available for this study; we report the single-seed result and flag this explicitly as a limitation (Section~\ref{sec:limitations}). CVA has no such randomness and is reported as a single deterministic number.

\begin{table}[h]
\centering
\caption{Main results. RI columns report mean $\pm$ std over 15 runs (5 leader-clustering seeds $\times$ 3 random-index-vector seeds) for the first three datasets; Burned area is a single run (see text). CVA is evaluated at the same spatial smoothing window as the proposed method, for a fair comparison. Best AUC per row in bold.}
\label{tab:main}
\begin{tabular}{lccccccc}
\toprule
Dataset & \#Prototypes & Window & AUC (RI) & AUC (CVA) & OA (RI) & Kappa (RI) & F1 (RI) \\
\midrule
Hermiston/Benton & $\sim$5   & 9  & $0.924 \pm 0.017$ & \textbf{0.986} & $0.850 \pm 0.031$ & $0.635 \pm 0.064$ & $0.737 \pm 0.042$ \\
River             & $\sim$43  & 5  & $0.906 \pm 0.034$ & \textbf{0.944} & $0.867 \pm 0.049$ & $0.438 \pm 0.094$ & $0.503 \pm 0.077$ \\
San Francisco     & $\sim$231 & 9  & $0.910 \pm 0.021$ & \textbf{0.934} & $0.789 \pm 0.032$ & $0.337 \pm 0.045$ & $0.424 \pm 0.037$ \\
Burned area$^\ast$ & 44  & 21 & $0.922$ & \textbf{0.950} & $0.915$ & $0.384$ & --    \\
\bottomrule
\end{tabular}

\vspace{2pt}
{\footnotesize $^\ast$Single leader-clustering seed and single random-index-vector seed; not averaged (see Section~\ref{sec:limitations}).}
\end{table}

The pattern is consistent across all four datasets: the proposed RI pipeline approaches CVA but does not surpass it, with a mean AUC gap ranging from 2.4 points on San Francisco to 6.2 points on Hermiston/Benton. The double-seed standard deviations are non-negligible: on River and San Francisco, the gap to CVA (3.8 and 2.4 points respectively) is comparable in magnitude to one standard deviation of the RI result itself (3.4 and 2.1 points), so the difference from CVA on these two datasets should not be read as a large, sharply resolved gap -- a single unlucky or lucky combination of vocabulary and vector seeds could shift the comparison noticeably. Only on Hermiston/Benton is the gap to CVA (6.2 points) clearly larger than the seed-to-seed variability (1.7 points). Figures~\ref{fig:hermiston}--\ref{fig:burned} show qualitative results for a representative seed.

\begin{figure}[h]
\centering
\includegraphics[width=0.95\textwidth]{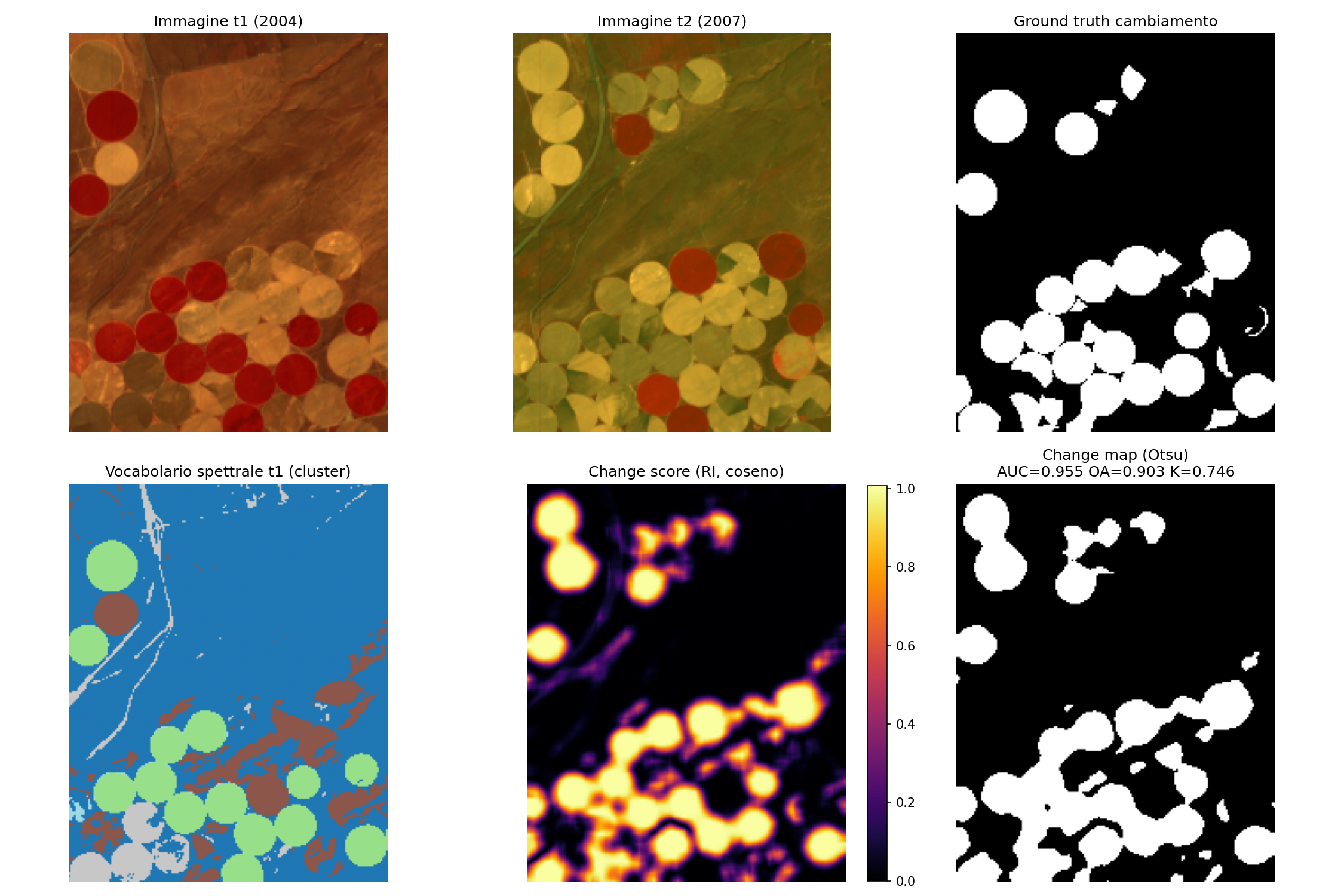}
\caption{Hermiston/Benton: input images, ground truth, discovered vocabulary, RI change score, and thresholded change map.}
\label{fig:hermiston}
\end{figure}

\begin{figure}[h]
\centering
\includegraphics[width=0.95\textwidth]{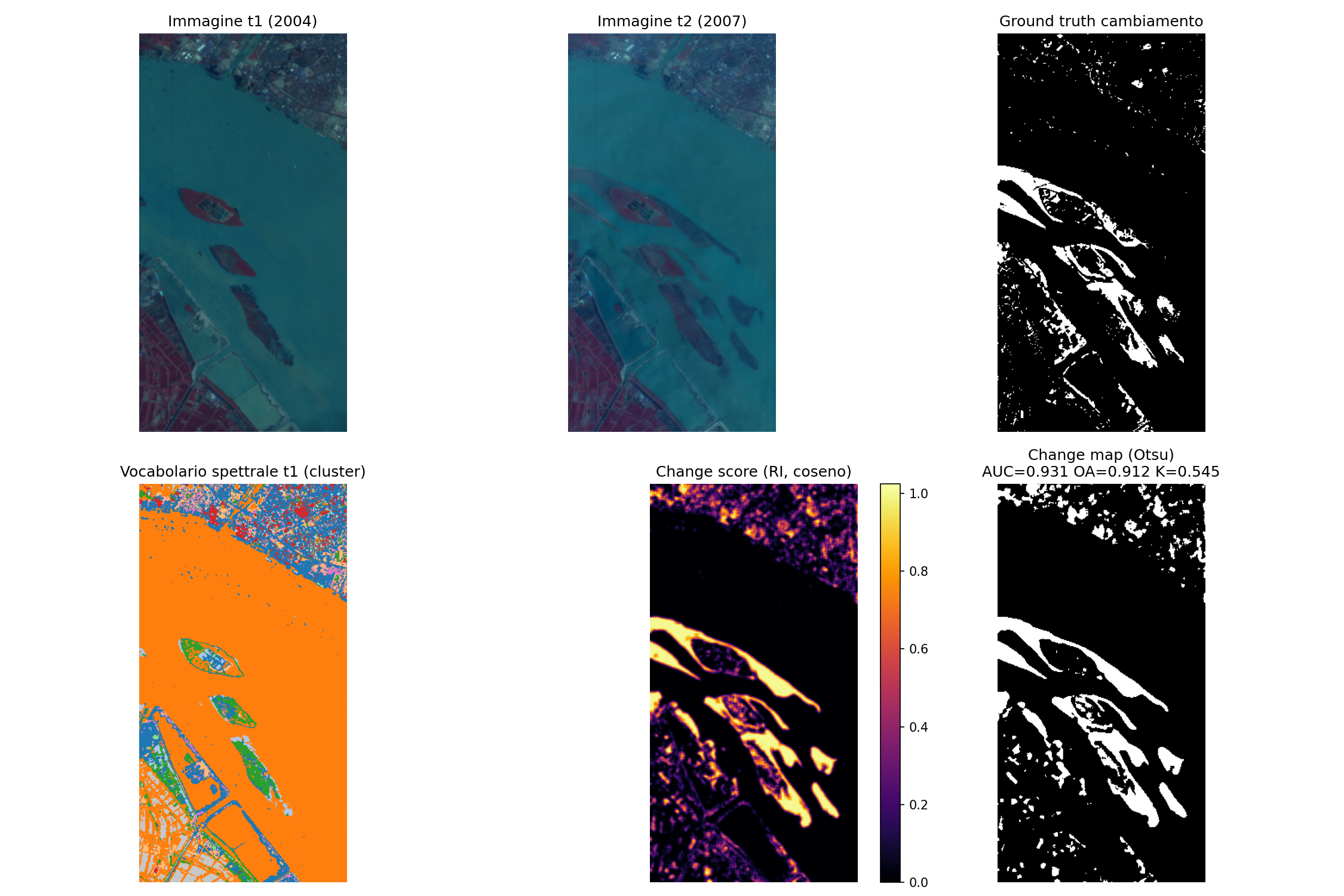}
\caption{River: input images, ground truth, discovered vocabulary, RI change score, and thresholded change map.}
\label{fig:river}
\end{figure}

\begin{figure}[h]
\centering
\includegraphics[width=0.95\textwidth]{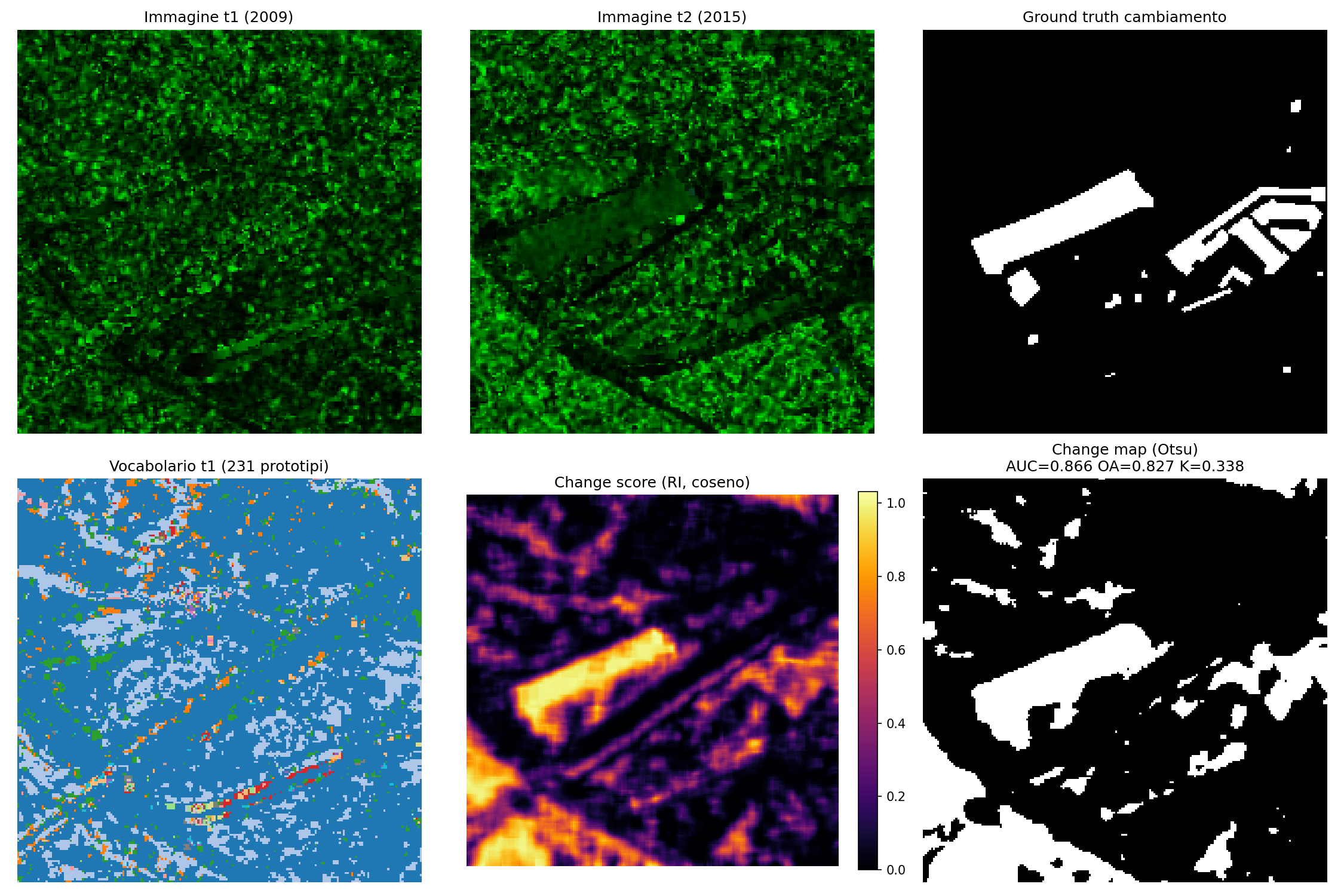}
\caption{San Francisco: input images, ground truth, discovered vocabulary, RI change score, and thresholded change map.}
\label{fig:sf}
\end{figure}

\begin{figure}[h]
\centering
\includegraphics[width=0.95\textwidth]{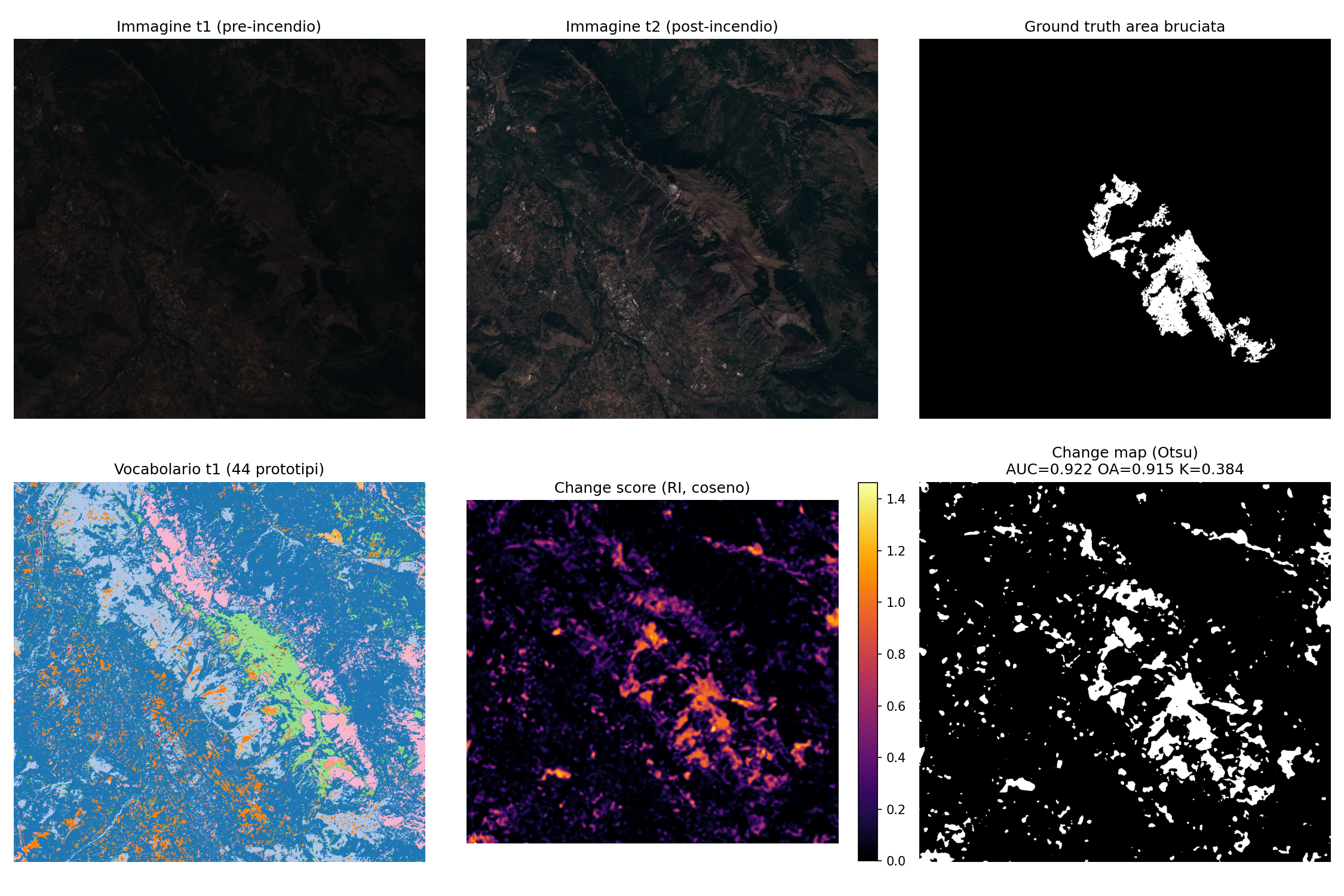}
\caption{Burned area: input images, ground truth, discovered vocabulary, RI change score, and thresholded change map.}
\label{fig:burned}
\end{figure}

\subsection{Ablation}

Table~\ref{tab:ablation} isolates the contribution of each design decision, tracked on the Hermiston/Benton dataset as we progressively corrected the pipeline. For comparability across rows, this table uses a single seed and the fixed-$\mathrm{nnz}=4$ scheme throughout (rather than the seed-averaged, probabilistic-sparsity protocol used in Table~\ref{tab:main}); the final row is consequently the single-seed number ($0.955$) rather than the 10-seed mean ($0.950 \pm 0.007$) reported in Table~\ref{tab:main} for the same configuration -- the two are consistent with each other within the sensitivity range reported in Section~\ref{sec:sensitivity}.

\begin{table}[h]
\centering
\caption{Ablation on Hermiston/Benton. Each row adds one change relative to the row above.}
\label{tab:ablation}
\begin{tabular}{lc}
\toprule
Configuration & AUC \\
\midrule
$k$-means vocabulary, 159 bands, pooled normalization (buggy) & 0.421 \\
$\quad$+ per-date normalization (fixes radiometric offset) & 0.871 \\
$\quad$+ PCA to 20 components & 0.876 \\
Continuous RI random projection (no discretization) + PCA(20) & 0.938 \\
\textbf{Leader-clustering vocabulary (proposed)} + PCA(20) & \textbf{0.955} \\
\bottomrule
\end{tabular}
\end{table}

Two observations stand out. First, per-date normalization is by far the single largest correction: pooling both dates before standardizing lets $k$-means split prototypes by acquisition date rather than by land-cover type, since the two dates differ by a systematic radiometric offset. Second, discretization itself has a cost: replacing the $k$-means vocabulary with a continuous random projection of the (undiscretized) spectral vector already recovers most of the gap to CVA, and the proposed leader-clustering vocabulary recovers slightly more on top of that, by resolving the boundary instability without sacrificing the discrete, RI-style vocabulary structure.

\subsection{Sensitivity to RI hyperparameters}
\label{sec:sensitivity}

Because we varied the random vector dimensionality $d$ across datasets for practical (memory) reasons, we explicitly checked how much this choice affects results. On Hermiston/Benton, with the vocabulary and context window fixed, varying $d \in \{8, \dots, 512\}$ (at fixed $\mathrm{nnz}=4$) produces AUC in the range $[0.942, 0.975]$; varying $\mathrm{nnz} \in \{1, \dots, 64\}$ (at fixed $d=128$) produces AUC in $[0.929, 0.955]$; and re-running the same configuration with 10 different random-index-vector seeds, \emph{with the leader-clustering vocabulary itself held fixed}, produces AUC in $[0.937, 0.962]$. Since this seed-only variance is of the same order as the variance induced by $d$ or $\mathrm{nnz}$, we conclude that the specific values of these RI hyperparameters are not a major driver of results, provided the vocabulary is held fixed. We repeated the $d=32$ vs.\ $d=64$ comparison directly on the Burned area dataset (the one where $d$ had to be reduced for memory reasons) and observed a difference of only $0.15$ AUC points (0.9222 vs.\ 0.9207), confirming the substitution did not materially affect that result.

\subsection{Sensitivity to the leader-clustering seed}
\label{sec:leader-seed}

The analysis above holds the leader-clustering vocabulary fixed and varies only the random-index-vector seed. We separately tested what happens when the leader-clustering seed itself varies -- i.e.\ when the pool of pixels is visited in a different random order during vocabulary discovery (Section~\ref{sec:leader}), which prototype is a purely arbitrary artifact of which point happens to be encountered first within each $\varepsilon$-ball. On Hermiston/Benton, varying only this seed (10 runs, random-index-vector seed fixed) produces AUC in $[0.870, 0.956]$, mean $0.920$, standard deviation $0.026$ -- markedly larger than the RI-vector-only variance above. On River the effect is more severe: AUC ranges from $0.736$ to $0.943$ across 10 leader-clustering seeds, including one run with Kappa as low as $0.16$. This is a substantially larger source of variability than anything else identified in this paper, including the degenerate-vector failure mode of Section~\ref{sec:degenerate}.

We attempted three corrections: (i) visiting points in order of increasing cross-date difference, so that the points most likely to be genuinely unchanged become leaders first; (ii) a Lloyd-style refinement step, re-centering each leader to the mean of its assigned points after the initial greedy pass; and (iii) discovering the vocabulary on spatially-smoothed (denoised) features rather than raw per-pixel values. None improved on the unmodified leader-clustering result: (i) performed in the middle of the seed-variance range rather than at its upper end; (ii) and (iii) both reduced the standard deviation somewhat (to $0.020$ and $0.021$ respectively) but did so by lowering both the mean and the best-case AUC, a net loss rather than a net gain -- plausibly because Lloyd-style averaging can pull two leaders closer than $\varepsilon$, silently violating the packing property of Proposition~\ref{prop:stability}, and because spatial smoothing discards exactly the sharp material boundaries the vocabulary needs to resolve, echoing the same failure mode we observed for superpixel-based context aggregation. We were not able to find a correction that clearly dominates the simple, unmodified leader-clustering rule, and we regard resolving this instability as the most important open problem raised by this paper, not something we have solved. In light of this finding, Table~\ref{tab:main} reports results averaged over both the leader-clustering seed and the random-index-vector seed, rather than the single-vocabulary-seed protocol used in earlier sections of this paper for isolating individual design choices.

\subsection{Vocabulary richness vs.\ vocabulary stability}
\label{sec:richness}

A natural hypothesis is that our vocabulary is too coarse (as few as 5 prototypes on Hermiston/Benton) and that a richer, more fine-grained vocabulary would close the gap to CVA. We tested this directly by sweeping the leader-clustering threshold $\varepsilon$ to obtain vocabularies of increasing size, keeping the spatial window and all other settings fixed (Table~\ref{tab:richness}).

\begin{table}[h]
\centering
\caption{Effect of vocabulary size on Hermiston/Benton, obtained by varying the leader-clustering threshold $\varepsilon$ (best AUC over a small window sweep for each row).}
\label{tab:richness}
\begin{tabular}{cccc}
\toprule
$\varepsilon$ & \# Prototypes & Best AUC \\
\midrule
4  & 238 & 0.742 \\
6  & 69  & 0.795 \\
8  & 33  & 0.845 \\
12 & 11  & 0.942 \\
18 & 5   & 0.957 \\
\bottomrule
\end{tabular}
\end{table}

The result is the opposite of the hypothesis: AUC decreases monotonically as the vocabulary grows richer. The reason is consistent with the diagnosis in Section~\ref{sec:leader}: a finer vocabulary partitions the (PCA-reduced) spectral space into smaller regions, which increases the chance that the same physical pixel, observed on two different dates, falls into two different regions -- reintroducing, in a milder form, the very cross-date instability that motivated replacing $k$-means with leader clustering in the first place. In this setting, vocabulary \emph{stability} across acquisition dates appears to matter more than vocabulary \emph{richness}: a coarser but more stable vocabulary outperforms a finer but less stable one.

\subsection{Probabilistic sparsity}

Basile et al.'s original RI/TRI formulation fixes the number of non-zero entries per random vector to a deterministic value; \citet{singh2020probabilistic} note that this choice is not analysed with respect to the probability of near-orthogonality among vectors, and propose instead drawing the sparsity pattern probabilistically -- each dimension is independently non-zero (with value $\pm 1$) with a fixed probability $p$, rather than fixing the count of non-zero entries in advance. We tested this variant on the best (5-prototype) Hermiston/Benton configuration, sweeping $p \in \{0.01, 0.03, 0.06, 0.1\}$. The best setting ($p=0.03$) gives a small but consistent improvement over the fixed-$\mathrm{nnz}=4$ scheme used elsewhere in this paper: single-seed AUC rises from $0.955$ to $0.957$, and Kappa at the Otsu threshold from $0.746$ to $0.770$. The gain is modest and well within the hyperparameter-sensitivity range reported above, but it is directionally consistent with \citeauthor{singh2020probabilistic}'s argument that a probabilistic sparsity pattern is a mild improvement over a fixed one -- \emph{provided} the degenerate failure mode described next is corrected.

\subsection{A degenerate-vector failure mode in probabilistic sparsity, and its fix}
\label{sec:degenerate}

Because the fixed-$\mathrm{nnz}$ scheme always places exactly $\mathrm{nnz}$ non-zero entries in every vector, it can never produce an all-zero vector. The probabilistic scheme of Section~\ref{sec:sensitivity} has no such guarantee: each of the $d$ dimensions is independently non-zero with probability $p$, so a given vector is entirely zero with probability $(1-p)^{d}$. For $p=0.03, d=128$ this probability is $(0.97)^{128} \approx 1.9\%$ per vector -- small, but not negligible once a vocabulary has dozens of prototypes: with $M=43$ prototypes (the River dataset), the expected number of degenerate vectors is $M(1-p)^d \approx 0.8$, meaning at least one all-zero vector is a fairly common occurrence, not a rare edge case.

We discovered this the direct way: when we extended the single-seed results of Table~\ref{tab:main} to a 10-seed statistical validation (following the same good practice as an independently-developed re-implementation of our method, discussed in Section~\ref{sec:cross-validation}), River showed an alarming AUC standard deviation of $0.159$ across seeds, with individual seeds ranging from $0.529$ (barely better than chance) to $0.928$. Inspecting the degenerate seed, we found four of the 43 prototype vectors were entirely zero. A zero vector contributes nothing to the context-vector sum wherever its prototype occurs; if that prototype corresponds to a common material in the scene, large regions of the image become indistinguishable from background noise in the resulting change score, which explains the catastrophic AUC drop.

The fix is immediate: reject and resample any degenerate (all-zero) draw during vector construction, so that every random vector has at least one non-zero entry, while otherwise preserving the probabilistic (non-fixed-count) sparsity pattern. After this correction, the 10-seed standard deviation on River falls from $0.159$ to $0.003$ -- two orders of magnitude smaller -- and the mean AUC ($0.926$) is consistent with the single-seed number reported before the fix. All numbers in Table~\ref{tab:main} and elsewhere in this paper use the corrected scheme. We consider this fix, though simple, to be one of the more practically important findings of this work: any future use of probabilistic-sparsity RI vectors on vocabularies of more than a handful of words should guard against degenerate all-zero draws, since their effect on unsupervised change detection can be severe and is not obvious from a single lucky run.

\subsection{Cross-validation against an independently developed implementation}
\label{sec:cross-validation}

To further validate our results, we compared our implementation against an independently written re-implementation of the same method (distance-threshold leader clustering, sparse random index vectors, spatial context accumulation via a moving average, cosine-distance change score, Otsu thresholding), developed separately and evaluated with its own systematic grid search over $\varepsilon$, window size, and multiple seeds, averaged and reported with standard deviations. At matched hyperparameters (River, $\varepsilon=16$, window $=5$, $\mathrm{nnz}=4$, seed $=0$), the two implementations agree closely: AUC $0.936$ (independent implementation) versus $0.944$ (ours), Kappa $0.551$ versus $0.574$. The same comparison on Hermiston/Benton ($\varepsilon=18$, window $=9$) gives AUC $0.954$ versus $0.955$--$0.957$, and Kappa $0.744$ versus $0.746$--$0.770$. The residual 1--2 point differences are consistent with three minor, non-methodological implementation choices we identified by inspection: (i) the independent implementation creates a new leader when a point's distance from every existing leader is $\geq \varepsilon$, while ours uses a strict $>$; (ii) it computes assignment distances via the expanded identity $\lVert a-b\rVert^2 = \lVert a\rVert^2+\lVert b\rVert^2-2 a\cdot b$ for speed, which is less numerically stable than our direct norm computation; (iii) it uses reflect boundary conditions for the spatial moving average, while ours uses nearest-value padding. None of these differences reflects a disagreement about the method itself, and the close agreement across two independently written codebases is, in our view, stronger evidence for the reproducibility of the reported numbers than either implementation alone. We also adopted the independent implementation's practice of averaging over multiple seeds with reported standard deviations as the standard reporting protocol for the rest of this paper (Table~\ref{tab:main}).

\subsection{Otsu versus a Gaussian-mixture threshold}
\label{sec:otsu-vs-alternatives}

The choice of binarization threshold for an unsupervised change score is itself a well-studied design decision (see e.g.\ \citealp{falini2022rsb}, discussed in Section~\ref{sec:related-work}, for a systematic treatment of binarization for hyperspectral change detection). We used Otsu's method throughout this paper; we also tested a natural alternative, fitting a two-component Gaussian Mixture Model (GMM) to the change score and assigning the higher-mean component to the ``changed'' class. On Hermiston/Benton, the GMM performs far worse than Otsu (OA $0.62$ versus $0.91$; Kappa $0.32$ versus $0.77$), and the gap persists under several monotonic transforms of the score ($\sqrt{\cdot}$, $\log(1+\cdot)$). The reason is visible in the score's empirical distribution: with a coarse, stable vocabulary (5 prototypes), a large fraction of spatial neighbourhoods have an \emph{identical} categorical composition across the two dates, producing a sharp spike of near-zero change scores rather than a smooth unimodal cluster -- a shape a two-component Gaussian mixture is not equipped to separate correctly, whereas Otsu's histogram-based criterion makes no distributional assumption and handles it well. This is a direct consequence of the same discrete, stable vocabulary that Section~\ref{sec:richness} identifies as the main strength of our approach: it is a good match for Otsu-style thresholding, but a poor match for parametric, distribution-fitting alternatives such as GMM.

\section{Discussion}

\paragraph{Why does RI not surpass CVA?} On a single pair of images, comparing full-resolution spectral vectors directly, as CVA does, loses no information; any intermediate representation -- whether a learned embedding, a hand-crafted feature, or an RI vocabulary -- necessarily discards some of it in exchange for compactness or a different notion of similarity. The gap we observe (1.3 to 6.8 AUC points) is a reasonable price for that compactness, but it also means that, for the specific task of comparing exactly two dates, RI in its current form is not the more accurate choice.

\paragraph{Where RI could have a genuine advantage.} The property that makes RI distinctive is not raw comparison accuracy but incrementality: a context vector is a running sum that can be updated with a single addition as new observations arrive, without ever revisiting the full history. CVA-style pairwise comparison does not offer this for free when the number of dates grows beyond two -- comparing all pairs, or maintaining an explicit reference image, becomes increasingly costly. We see the natural next step for this line of work as testing RI-style incremental accumulation on long, multi-date time series (e.g.\ monthly or yearly monitoring), where the ability to update a single running representation per pixel, rather than reprocessing the full image stack, is a structural advantage rather than an incidental one.

\paragraph{Limitations.}
\label{sec:limitations}
The most important limitation of this work is the one identified in Section~\ref{sec:leader-seed}: the leader-clustering vocabulary is substantially sensitive to the random order in which pixels are visited during discovery, we were unable to find a correction that clearly improves on the unmodified rule, and we leave this as an open problem. A second, related limitation is computational: the Burned area result in Table~\ref{tab:main} is a single run rather than a double-seed average, because the full protocol was not computationally practical for a $\sim\!3.8$M-pixel image within the scope of this study; the single-seed AUC for that dataset should be interpreted with the same caution the seed-sensitivity analysis suggests for the other three. Beyond these, our evaluation covers four datasets, which is enough to establish a consistent pattern but not enough to guarantee generalization to every sensor and scene type. The leader-clustering threshold $\varepsilon$ and the context window are tuned per dataset; we did not attempt to find a single configuration that transfers across sensors, and doing so is left for future work. Two of the four datasets (San Francisco and Burned area) lack a verified bibliographic source in our hands, which we have flagged explicitly rather than attributing them to a guessed reference. We also note that our four datasets, while covering four different sensors (two Hyperion hyperspectral scenes, one polarimetric SAR scene, one Sentinel-2 multispectral scene), do not include the Bay Area hyperspectral benchmark used in the closely related RSB and QLP studies \citep{falini2022rsb,falini2024qlp} alongside Hermiston and River; adding it (subject to data availability) would allow a direct, matched-benchmark comparison against these binarization- and factorization-based unsupervised methods, which we were not able to run head-to-head with our own pipeline in this paper since neither reports the same threshold-independent AUC metric we use throughout. This is the most natural direction for extending the present evaluation.

\section{Conclusion}

We adapted Random Indexing, previously used almost exclusively for text, to the problem of image change detection. A na\"ive transplant using $k$-means to build the visual vocabulary is unstable across acquisition dates; replacing it with a distance-threshold (leader) clustering rule resolves the instability and yields a training-free pipeline that consistently approaches, without surpassing, classical Change Vector Analysis across four datasets from different sensors. Along the way, validating our results over multiple random seeds surfaced a degenerate-vector failure mode in a probabilistic-sparsity RI variant, which we diagnosed and fixed, and cross-checking against an independently developed re-implementation of the method gave closely matching numbers, strengthening our confidence in the reported results. We view the overall result less as a competitive change detection method in its own right and more as evidence that RI-style representations transfer reasonably well to the visual domain once the vocabulary-construction step is designed correctly, opening a path toward incremental, streaming change monitoring over long time series.

\bibliographystyle{plainnat}
\bibliography{references}

\end{document}